\documentclass[reqno]{amsart}
\newcommand{\hoje}{}
\renewcommand{\hoje}{\footnotesize{\number\day-\number\month-\number\year \ }}
\usepackage{verbatim}
\usepackage{pstricks}
\usepackage{graphicx}
\usepackage{amsmath,amssymb,amstext,amsthm,amsfonts}
   
   \DeclareMathOperator{\e}{e}

\usepackage{euscript}

\usepackage{enumerate}
\newtheorem{definition}{Definition}
\newtheorem{theorem}{Theorem}
\newtheorem{lemma}{Lemma}
\newtheorem{corollary}{Corollary}

\renewcommand{\epsilon}{\varepsilon}

\newcommand{\lbd}{\lambda}

\newcommand{\prts}[1]{\left(#1\right)}

\newcommand{\abs}[1]{\left|#1\right|}

\newcommand{\set}[1]{\left\{#1\right\}}

\newcommand{\ov}[1]{\overline{#1}}
\newcommand{\til}[1]{\widetilde{#1}}
\usepackage{dsfont}
   \newcommand{\N}{\mathbb{N}}
   \newcommand{\R}{\mathbb{R}}
   \newcommand{\Z}{\mathbb{Z}}

\newcommand{\Es}{\hspace{0.5cm}}
\newcommand{\dst}{\displaystyle}
\newcommand{\er}{\mathbb{R}}
\newcommand{\en}{\mathbb{N}}

\newcommand{\ze}{\mathbb{Z}}

\renewcommand{\le}{\leqslant}
\renewcommand{\leq}{\leqslant}
\begin{document}
\title[\hoje Stability of difference equations with delays]
   { Existence and stability of a periodic solution of a general difference equation with applications to neural networks with a delay in the leakage terms} 
\author[\hoje Ant\'onio J. G. Bento]{Ant\'onio J. G. Bento}
\address{A. Bento\\
   Departamento de Matem\'atica\\
   Universidade da Beira Interior\\
   6201-001 Covilh\~a\\
   Portugal}
\email{bento@mat.ubi.pt}
\author{Jos\'e J. Oliveira}
\address{J. Oliveira\\
   Departamento de Matem\'atica\\
Escola de Ci\^{e}ncias\\ Universidade do Minho\\ Campus de Gualtar\\ 4710-057 Braga\\ Portugal}
\email{jjoliveira@math.uminho.pt}
\author{C\'esar M. Silva}
\address{C. Silva\\
   Departamento de Matem\'atica\\
   Universidade da Beira Interior\\
   6201-001 Covilh\~a\\
   Portugal}
\email{csilva@mat.ubi.pt}
\urladdr{www.mat.ubi.pt/~csilva}
\date{\today}
\subjclass{}
\keywords{}
\begin{abstract}
  In this paper, a new global exponential stability criterion is obtained for a general multidimensional delay difference equation using induction arguments. In the cases that the difference equation is periodic, we prove the existence of a periodic solution by constructing a type of Poincar\'e map. The results are used to obtain stability criteria for a general discrete-time neural network model with a delay in the leakage terms. As particular cases, we obtain new stability criteria for neural network models recently studied in the literature, in particular for low-order and high-order Hopfield and Bidirectional Associative Memory(BAM).   
 \end{abstract}
\maketitle
\section{Introduction}
Neural network models have an important role in several scientific areas, such as geology \cite{pham_nguyen_bui_prakash_chapi_bui}, medicine \cite{lu}, and physics \cite{cai_mao_wang_yin_karniadakis}, due to their power to be applied in sign and image processing, pattern recognition, optimization problems, and so on \cite{aizenberga_aizenberg_hiltner_moraga_Bexten,cichocki,pham_nguyen_bui_prakash_chapi_bui,velichko_belyaev_boriskov}. Therefore, since the pioneer works of Cohen and Grossberg \cite{cohen-grossberg}, Hopfield \cite{hopfield}, and Kosko \cite{kosko}, many researchers devoted themselves to study the dynamic behavior of neural network models.

The artificial neural network models studied in \cite{cohen-grossberg,hopfield,kosko} are described by ordinary differential equations. However, to take into account the transmission speed of signals between different neurons, it is essential to introduce delays in the models. Marcus and Westervelt \cite{marcus+westervelt} included a discrete delay in the model studied by Hopfield \cite{hopfield} and observed that delays induce instability in its dynamical behavior (see also \cite{baldi}). Later, Gopalsamy \cite{gopalsamy} introduced discrete delays in the negative feedback terms of a continuous-time BAM model. These terms are also called ``forgetting" or leakage terms \cite{kosko}. Since then, the stability of continuous-time neural network models with delays in the leakage terms has been the subject of study by an increasing number of researchers (see \cite{berezansky+braverman+idels,liu,jj,peng,Thoiyab+muruganantham+zhu+gunasekaran} and references therein). 

By computational reasons, the discrete-time models are better digitally implemented than continuous-time ones  \cite{Mohamad_Gopalsamy-AMC-2003}, thus it is important to study the stability of discrete-time neural network models. There are several works concerning discrete-time neural networks \cite{Bento-oliveira-Silva,Chen_Song_Zhao_Liu,dai_du,Dong_Wang_Zhang-AMC-2020,Dong_Wang_Zhang_hu_dinh-NA-2023,Hong_Ma-MBE-2019,Mohamad_Gopalsamy-AMC-2003,jj2,Qui_Liu_Shu-ADE-2016,Raja_Anthoni-CNSNS-2011,Sowmiya_Raja_Cao_Li_Rajchakit-JFI-20,Sowmiya_Raja_Cao_Rajchakit_Alsaedi-AJC-2020,Shumin_Yanhong,Suntonsinsoungvon_Udpin,xu_wu,zheng_du} but, as far as we know, just a few works have been dedicated to the study of discrete-time neural networks with delays in the leakage terms \cite{Chen_Song_Zhao_Liu,jj2,Sowmiya_Raja_Cao_Li_Rajchakit-JFI-20,Sowmiya_Raja_Cao_Rajchakit_Alsaedi-AJC-2020,Suntonsinsoungvon_Udpin}. In \cite{Chen_Song_Zhao_Liu,jj2} the global stability was studied for Hopfield type models with delays in the leakage terms and, in \cite{Sowmiya_Raja_Cao_Li_Rajchakit-JFI-20}, for a stochastic impulsive BAM model also with delays in the leakage terms. However, all models studied in \cite{Chen_Song_Zhao_Liu,jj2,Sowmiya_Raja_Cao_Li_Rajchakit-JFI-20} have constants  parameters such as neuron charging time, interconnection weights, and external inputs. In \cite{Sowmiya_Raja_Cao_Rajchakit_Alsaedi-AJC-2020} an asymptotic stability criterion for an uncertain BAM model with delays in the leakage terms was established, while in \cite{Suntonsinsoungvon_Udpin}  an exponential stability criterion for an uncertain Hopfield model with delays in the leakage terms was obtained.

Models become more realistic if changes in parameters over time are considered, thus it is important to study nonautonomous neural network models. As far as we know, the global stability of discrete-time neural network models with delays in the leakage terms and changes of parameters over the time has not yet been studied.  In particular, all  periodic or almost periodic models studied have no delays in the leakage terms \cite{Bento-oliveira-Silva,dai_du,Shumin_Yanhong,xu_wu,zheng_du}. 

In the present work, we establish a sufficient condition for the global exponential stability of the following general $N$-dimensional delay difference equation
  \begin{eqnarray}\label{discrete-general-model-introducao}
  	x_i(m+1)
  	= c_i(m)x_i(m-\tau)+h_i\left(m,\overline x_m\right),\Es m\in\en_0,\,i\in\{1,\ldots,N\},
  \end{eqnarray}
and, in the case of \eqref{discrete-general-model-introducao} being a periodic equation, we prove the existence of a periodic solution as a consequence of its global exponential stability.

In this paper, we only apply our abstract results to neural network type models. Despite, equation \eqref{discrete-general-model-introducao} is general enough to include other type of discrete-time models such as the hematopoiesis type models with monotone production rate \cite{li+zhang}. The proof presented here to establish our main stability result involves induction arguments, which is a new method to prove the stability of difference equations. In fact, the proofs usually present in the literature involve the construction of a suitable Lyapunov function \cite{Chen_Song_Zhao_Liu,gopalsamy,liu,Mohamad_Gopalsamy-AMC-2003,peng,Raja_Anthoni-CNSNS-2011,Sowmiya_Raja_Cao_Li_Rajchakit-JFI-20,Sowmiya_Raja_Cao_Rajchakit_Alsaedi-AJC-2020,Thoiyab+muruganantham+zhu+gunasekaran,xu_wu}, or some Halanay inequalities \cite{Suntonsinsoungvon_Udpin}, or other type of inequality analysis techniques \cite{dai_du,Dong_Wang_Zhang-AMC-2020,Dong_Wang_Zhang_hu_dinh-NA-2023,Hong_Ma-MBE-2019,jj2,zheng_du}. 

The main novelties in this work are:
\begin{enumerate}
	\item The global exponential stability criterion, Theorem \ref{thm:stability-criterion-discrete-model}, established for the nonautonomous difference equation  \eqref{discrete-general-model-introducao}. We emphasize the new method used in the proof and the fact that we are dealing with a difference equation with a delay in the linear part.
	\item The global exponential stability criterion, Theorem \ref{teo:est-modelo-geral-aplicacoes-discreto}, established for a discrete-time generalized neural network model with delay in the leakage terms \eqref{eq:modelo-geral-aplicacao}, which is general enough to include several neural network models as particular cases. We note that the low-order Hopfield type model \eqref{eq:modelo-Hopfield-geral}, the BAM model \eqref{eq:BAM}, and the high-order Hopfield type model \eqref{eq:hopfield-high-order} are particular cases of \eqref{eq:modelo-geral-aplicacao}.
	\item The criterion for the existence of periodic solutions of \eqref{discrete-general-model-introducao}, Theorem \ref{teo:exist-est-sol-periodic-model-geral}. This result, together with Theorem \ref{teo:est-modelo-geral-aplicacoes-discreto}, allowed us to assure the existence and global exponential stability of a periodic solution of  periodic Hopfield type and BAM models with delay in the leakage terms, Corollaries \ref{cor:exist-est-sol-periodic-model-hopfiels-periodico},  \ref{cor:exist-est-sol-periodic-model-BAM-autonomo}, and  \ref{cor:exist-est-sol-periodic-model-HOHM-periodico}. Previously, the existence and global exponential stability of a periodic solution were established for discrete-time periodic neural network models without delays in the leakage terms \cite{Bento-oliveira-Silva,dai_du,xu_wu}. 
\end{enumerate} 
This paper is organized into five sections. After the introduction, in Section \ref{main-results} the main global stability criterion of \eqref{discrete-general-model-introducao} is proved. In Section \ref{periodic-models}, we assume that  \eqref{discrete-general-model-introducao} is a periodic difference equation and, considering a Poincar\'e map, we obtain the existence of a periodic solution as a consequence of the global exponential stability.  In Section \ref{applications}, we apply the results in Section \ref{main-results} and \ref{periodic-models} to obtain stability criteria and the existence of periodic solutions of nonautonomous discrete-time neural network models with delay in the leakage terms. In this section, a comparison of our results with the ones in the literature is done. Finally, in Section \ref{numericalexample}, a numerical example is given to illustrate the effectiveness of some of our results.

\section{Stability of Models with delay in the leakage term}\label{main-results}

Consider the difference equation 
\begin{eqnarray}\label{discrete-general-model}
x_i(m+1)
= c_i(m)x_i(m-\tau)+h_i\left(m,\overline x_m\right),\Es m\in\en_0,\,i\in\{1,\ldots,N\},
\end{eqnarray}
with $N\in\en$, $\tau\in\en_0$, and, for each $i\in\{1,\ldots,N\}$,  $c_i:\en_0\to]-1,1[$ and $h_i:\en_0\times X^N\to\er$ are functions. The space $X^N$ and the notation $\overline x_m$ are explained below.

Given a set $I\subseteq\er$, define $I_\ze=I\cap\ze$. Consider $r\in]-\infty,-\tau]_\ze$ and denote by $X$ the space of the functions 
$$
 \alpha:[r,0]_\Z \to \R
$$
equipped with the norm
$$ 
  \|\alpha\| = \max\limits_{j=r, \ldots,0} |\alpha(j)|.
$$
We denote by $X^N$ and $\R^N$ the cartesian products equipped with the supremum norm, i.e., for $\ov\alpha = \prts{\alpha_1, \ldots, \alpha_N} \in X^N$ and $\ov y = \prts{y_1, \ldots, y_N} \in \R^N$, we have
$$ \|\ov\alpha\|
= \max_{i=1, \ldots,N} \|\alpha_i\|
= \max_{i=1, \ldots,N} \prts{\max_{j=r, \ldots, 0} \abs{\alpha_i(j)}}$$
and
$$ |\ov y| = \max_{i=1, \ldots, N} |y_i|.$$
We write $\ov{y}=(y_1,\ldots,y_N)>0$ in case of $y_i>0$ for all $i\in\{1,\ldots,N\}$.

Given a function $\ov x \colon [r,+\infty[_\Z \to \R^N$ we denote the $i$th component by $x_i$, i.e., $\ov x = \prts{x_1, \ldots x_N}$. For each $m \in \N_0$, we define $\ov x_m \in X^N$ by
$$ \ov x_m(j) = \ov x(m+j), \ j=r, r+1, \ldots, 0.$$

For each $\ov\alpha \in X^N$, we denote by $\ov x (\cdot,\ov\alpha)$ the unique solution
$$ \ov x:[r,+\infty[_\Z \to \R^N$$
of~\eqref{discrete-general-model} with initial conditions $\ov x_0 = \ov\alpha$.

The main purpose in this section is to establish sufficient conditions for the global stability of difference equation \eqref{discrete-general-model}.

\begin{definition}
	
	We say that difference equation \eqref{discrete-general-model} is:
	\begin{enumerate}
		\item  uniformly stable if
	$$
	  \forall \varepsilon>0,\exists\delta>0,\forall\ov\alpha,\ov\beta\in X^N,\forall m\in\en_0: \|\ov\alpha-\ov\beta\|<\delta\Rightarrow\|\overline{x}_m(\cdot,\ov\alpha)-\overline{x}_m(\cdot,\ov\beta)\|<\varepsilon.
	$$
	\item globally exponentially stable if 
	there are $C>0$, and $\zeta\in]0,1[$ such that, for all $\ov\alpha,\ov\beta\in X^N$,
	$$
	\|\ov{x}_m(\cdot,\ov\alpha)-\ov{x}_m(\cdot,\ov\beta)\|\leq C\zeta^{ m}\|\ov\alpha-\ov\beta\|,\,\,\forall m\in\en_0.
	$$ 
	\item globally attractive if
	$$
	  \lim_m\|\ov{x}_m(\cdot,\ov\alpha)-\ov{x}_m(\cdot,\ov\beta)\|=0,\,\,\forall \ov\alpha,\ov\beta\in X^N.
	$$ 
	\item globally asymptotically stable if it is uniformly stable and globally attractive.
	\end{enumerate}
\end{definition}

 For each $s\in\{0,\ldots,\tau\}$, we denote by $[s]_\tau$ the set
\begin{eqnarray}\label{definicao-classe-equivalencia}
  [s]_\tau:=\{n(\tau+1)-s: n\in\en\}.
\end{eqnarray}
Consequently, the equality $\en=\dst\bigcup_{s=0}^\tau[s]_\tau$ holds and $[s]_\tau\cap[s^*]_\tau=\emptyset$ if $s\neq s^*$ with $s,s^*\in\{0,\ldots,\tau\}$.

For the functions $c_i$ and $h_i$, we assume the following hypotheses
\begin{description}
	\item[(H1)] for each $i\in\{1,\ldots,N\}$, there is a function $H_i:\en_0\to]0,+\infty[$ such that
	$$
	 |h_i(m,\ov \alpha)-h_i(m,\ov \beta)|\leq H_i(m)\|\ov \alpha-\ov \beta\|,\,\,\,\forall \ov \alpha,\ov \beta\in X^N,\,m\in\N_0;
	$$
	\item[(H2)] there is $c\in]0,1]$ 
	such that $$
	|c_i(m)|\leq c,\,\forall i\in\{1,\ldots,N\},\,\forall m\in\en_0;
	$$
	and
	\begin{eqnarray}
	  \lambda:=\max_{i=1,\ldots,N}\max_{s=0,\ldots,\tau}\left[\sup_{n\in\en} \sum_{l=0}^{n-1}\right.\lefteqn{\left(\prod_{k=l+1}^{n-1}|c_i(k(\tau+1)+\tau-s)|\right)}\nonumber
	  \\
	  &\cdot H_i(l(\tau+1)+\tau-s)c^{l+\frac{r-s-1}{\tau+1}-n+1}\bigg]<1.\nonumber
	\end{eqnarray}
\end{description}

In hypothesis (H2), we use the standard convention that a product is equal to one if the number of factors is zero.
 
Before stating our main stability result, we need to prove the following lemma.

%
%
%

\begin{lemma}\label{lem-diferenca-solucoes}
	Let $\ov{\alpha}=(\alpha_1,\ldots,\alpha_N),\,\ov{\beta}=(\beta_1,\ldots,\beta_N)\in X^N$.
	
	If {\bf (H1)} holds, then the solutions $\ov x(\cdot,\ov{\alpha})=\ov x(\cdot)=(x_1(\cdot),\ldots,x_N(\cdot))$ and $\ov x(\cdot,\ov{\beta})=\ov y(\cdot)=(y_1(\cdot),\ldots,y_N(\cdot))$ of \eqref{discrete-general-model} verify
	\begin{eqnarray}\label{eq:diferenca-solucoes}
		\lefteqn{|x_i(n(\tau+1)-s)-y_i(n(\tau+1)-s)|}\nonumber\\
		&\leq\left(\dst\prod_{k=0}^{n-1}|c_i(k(\tau+1)+\tau-s)|\right)\|\ov\alpha-\ov\beta\|+\dst\sum_{l=0}^{n-1}\bigg(\prod_{k=l+1}^{n-1}|c_i(k(\tau+1)+\tau-s)|\bigg)\nonumber\\
		&\cdot H_i(l(\tau+1)+\tau-s)\|\ov x_{l(\tau+1)+\tau-s}-\ov y_{l(\tau+1)+\tau-s}\|,
	\end{eqnarray}
	for all  $i\in\{1,\ldots,N\}$, $s\in\{0,\ldots,\tau\}$, and $n\in\en$.
\end{lemma}
\begin{proof}
Let $s\in\{0,\ldots,\tau\}$ and  $i\in\{1,\ldots,N\}$. 

The proof is done using induction on $n\in\en$.
  
 For $n=1$, we have, from \eqref{discrete-general-model} and (H1),
     \begin{eqnarray}
     	\lefteqn{|x_i(\tau+1-s)-y_i(\tau+1-s)|}\nonumber\\
     	&\begin{array}{rl}
     		=&|c_i(\tau-s)x_i(\tau-s-\tau)+h_i(\tau-s,\ov x_{\tau-s})\\
     		&-c_i(\tau-s)y_i(\tau-s-\tau)-h_i(\tau-s,\ov y_{\tau-s})|\\
     \leq&|c_i(\tau-s)||x_i(-s)-y_i(-s)|+|h_i(\tau-s,\ov x_{\tau-s})-h_i(\tau-s,\ov y_{\tau-s})|\\
     \leq&|c_i(\tau-s)|\|\overline\alpha-\overline\beta\|+H_i(\tau-s)\|\ov x_{\tau-s}-\ov y_{\tau-s}\|\\
     =&\left(\dst\prod_{k=0}^{1-1}|c_i(k(\tau+1)+\tau-s)|\right)\|\ov\alpha-\ov\beta\|+\dst\sum_{l=0}^{1-1}\left(\prod_{k=l+1}^{1-1}|c_i(k(\tau+1)+\tau-s)|\right)\\
     &\cdot H_i(l(\tau+1)+\tau-s)\|\ov x_{l(\tau+1)+\tau-s}-\ov y_{l(\tau+1)+\tau-s}\|.
     \end{array}\nonumber
     \end{eqnarray}
  Now we assume that \eqref{eq:diferenca-solucoes} holds for some $n\in\en$. Consequently, from \eqref{discrete-general-model}, (H1) and induction hypothesis, we have
 \begin{eqnarray}
  \lefteqn{\big|x_i((n+1)(\tau+1)-s)-y_i((n+1)(\tau+1)-s)\big|}\nonumber\\
  &\begin{array}{rl}
  =&\big|x_i(n(\tau+1)+\tau-s+1)-y_i(n(\tau+1)+\tau-s+1)\big|\\
  =&\big|c_i(n(\tau+1)+\tau-s)x_i(n(\tau+1)-s)+h_i(n(\tau+1)+\tau-s,\ov x_{n(\tau+1)+\tau-s})\\
   &-c_i(n(\tau+1)+\tau-s)y_i(n(\tau+1)-s)-h_i(n(\tau+1)+\tau-s,\ov y_{n(\tau+1)+\tau-s})\big|\\
  \leq&|c_i(n(\tau+1)+\tau-s)||x_i(n(\tau+1)-s)-y_i(n(\tau+1)-s)|\\
   &+H_i(n(\tau+1)+\tau-s)\|\ov x_{n(\tau+1)+\tau-s}-\ov y_{n(\tau+1)+\tau-s}\|\\
  \leq&|c_i(n(\tau+1)+\tau-s)|\dst\left(\prod_{k=0}^{n-1}|c_i(k(\tau+1)+\tau-s)|\right)\|\ov{\alpha}-\ov{\beta}\|\\
   &\dst+|c_i(n(\tau+1)+\tau-s)|\sum_{l=0}^{n-1}\left(\prod_{k=l+1}^{n-1}|c_i(k(\tau+1)+\tau-s)|\right)\\
   &\cdot H_i(l(\tau+1)+\tau-s)\|\ov x_{l(\tau+1)+\tau-s}-\ov y_{l(\tau+1)+\tau-s}\|\\
   &\dst+H_i(n(\tau+1)+\tau-s)\|\ov x_{n(\tau+1)+\tau-s}-\ov y_{n(\tau+1)+\tau-s}\|\\
  =&\dst\left(\prod_{k=0}^{n}|c_i(k(\tau+1)+\tau-s)|\right)\|\ov{\alpha}-\ov{\beta}\|+\sum_{l=0}^{n}\left(\prod_{k=l+1}^{n}|c_i(k(\tau+1)+\tau-s)|\right)\\
   &\cdot H_i(l(\tau+1)+\tau-s)\|\ov x_{l(\tau+1)+\tau-s}-\ov y_{l(\tau+1)+\tau-s}\|
  \end{array}\nonumber
  \end{eqnarray}
  and the proof of \eqref{eq:diferenca-solucoes} is concluded. 
\end{proof}

Now we are in position to prove the main stability result about the difference equation \eqref{discrete-general-model}.

\begin{theorem}\label{thm:stability-criterion-discrete-model}
	Assume the hypotheses {\bf (H1)} and {\bf (H2)}.
	
	Then, for any $\ov{\alpha},\ov{\beta}\in X^N$, the solutions  $\ov x(\cdot,\ov{\alpha})$ and $\ov x(\cdot,\ov{\beta})$ of \eqref{discrete-general-model} verify
	\begin{equation}\label{eq:estabilidade-modelo-discreto1}
		\|\ov x_m(\cdot,\ov{\alpha})-\ov x_m(\cdot,\ov{\beta})\|
	\le \dfrac{c^{\frac{r}{\tau+1}-1}}{1-\lbd} \,\left(c^{\frac{1}{\tau+1}}\right)^m \|\ov\alpha-\ov{\beta}\|,
	\end{equation}
	for all $m\in\en_0$.	
\end{theorem}

\begin{proof}
  Let $\ov\alpha,\ov{\beta}\in X^N$ and consider the solutions 
   $\ov x(\cdot,\ov{\alpha})=\ov x(\cdot)=(x_1(\cdot),\ldots,x_N(\cdot))$ and $\ov x(\cdot,\ov{\beta})=\ov y(\cdot)=(y_1(\cdot),\ldots,y_N(\cdot))$ of \eqref{discrete-general-model}.
   
   First we prove that
   \begin{equation}\label{eq:estabilidade-modelo-discreto}
   	|\ov x(m,\ov{\alpha})-\ov x(m,\ov{\beta})|
   	\le \dfrac{1}{1-\lbd} \,c^{\frac{m-(\tau+1)}{\tau+1}}\|\ov\alpha-\ov{\beta}\|,\Es\forall m\in\en_0.
   \end{equation}
  The proof of inequality \eqref{eq:estabilidade-modelo-discreto} is done by induction on $m\in\en_0$. 
   
   As $0<c\leq1$ and $0<1-\lambda<1$, then condition \eqref{eq:estabilidade-modelo-discreto} holds  trivially for $m=0$.
   
    Assume that, for some $m\in\en$, we have
   \begin{eqnarray}\label{theo:hipotese-inducao}
   |\ov x(t)-\ov y(t)|
   \le \dfrac{1}{1-\lbd} \,c^{\frac{t-(\tau+1)}{\tau+1}} \|\ov\alpha-\ov{\beta}\|,\Es\forall t\in[0,m-1]_\ze.
   \end{eqnarray}
   
   As $\en=\dst\bigcup_{s=0}^\tau[s]_\tau$ and $[s]_\tau\cap[s^*]_\tau=\emptyset$ for $s\neq s^*$, by \eqref{definicao-classe-equivalencia} we conclude there are unique $s\in\{0,\ldots,\tau\}$ and $n\in\en$ such that $m=n(\tau+1)-s$. For $l\in\left[0,n-1\right]_\ze$, we have
   $$
     l(\tau+1)+\tau-s+j\leq(n-1)(\tau+1)+\tau-s=n(\tau+1)-s-1<m,\Es\forall j\in\{r,\ldots,0\},
   $$
   thus, by \eqref{theo:hipotese-inducao},

  \begin{align}\label{eq:estimativa-historico-solucao}
 \|\ov x_{l(\tau+1)+\tau-s}-\ov y_{l(\tau+1)+\tau-s}\|&=\max_{j=r,\ldots,0}|\ov x(l(\tau+1)+\tau-s+j)-\ov y(l(\tau+1)+\tau-s+j)|\nonumber\\
  &\leq \frac{1}{1-\lambda}\cdot\max_{j=r,\ldots,0}\left\{c^{\frac{l(\tau+1)+\tau-s+j-(\tau+1)}{\tau+1}}\right\}\|\ov{\alpha}-\ov{\beta}\|\nonumber\\
  &=\frac{1}{1-\lambda}c^{\frac{l(\tau+1)+\tau-s+r-\tau-1}{\tau+1}}\|\ov{\alpha}-\ov{\beta}\|\nonumber\\
  &=\frac{1}{1-\lambda}c^{l+\frac{r-s-1}{\tau+1}}\|\ov{\alpha}-\ov{\beta}\|.
  \end{align}
  
   Let $i\in\{1,\ldots,N\}$. From Lemma  \ref{lem-diferenca-solucoes} and \eqref{eq:estimativa-historico-solucao}, we have
   $$
   \begin{array}{rcl}
   |x_i(m)-y_i(m)|&=&|x_i(n(\tau+1)-s)-y_i(n(\tau+1)-s)|\\
   &\leq&\left(\dst\prod_{k=0}^{n-1}|c_i(k(\tau+1)+\tau-s)|\right)\|\ov\alpha-\ov\beta\|\\
   & &+\dst\sum_{l=0}^{n-1}\left(\prod_{k=l+1}^{n-1}|c_i(k(\tau+1)+\tau-s)|\right)\\
   &
   &\cdot H_i(l(\tau+1)+\tau-s)\|\ov x_{l(\tau+1)+\tau-s}-\ov y_{l(\tau+1)+\tau-s}\|\\
   &\leq&\left(\dst\prod_{k=0}^{n-1}|c_i(k(\tau+1)+\tau-s)|\right)\|\ov\alpha-\ov\beta\|\\
   & &+\dst\sum_{l=0}^{n-1}\left(\prod_{k=l+1}^{n-1}|c_i(k(\tau+1)+\tau-s)|\right)\\
   &
   &\cdot \dst H_i(l(\tau+1)+\tau-s)\frac{1}{1-\lambda}c^{l+\frac{r-s-1}{\tau+1}}\|\ov{\alpha}-\ov{\beta}\|\\
   &=&\left[\left(\dst\prod_{k=0}^{n-1}\frac{|c_i(k(\tau+1)+\tau-s)|}{c}\right)c\right.\\
   & & +\dst\frac{1}{1-\lambda}\sum_{l=0}^{n-1}\left(\prod_{k=l+1}^{n-1}|c_i(k(\tau+1)+\tau-s)|\right)H_i(l(\tau+1)+\tau-s)\\
   & &\cdot c^{l+\frac{r-s-1}{\tau+1}-n+1}\bigg]c^{n-1}\|\ov{\alpha}-\ov{\beta}\|.
   \end{array}
   $$
   Denoting 
    $$
   \mathcal{A}=\left(\prod_{k=0}^{n-1}\frac{|c_i(k(\tau+1)+\tau-s)|}{c}\right)c
   $$
   and
   $$ \mathcal{B}=\sum_{l=0}^{n-1}\left[\left(\prod_{k=l+1}^{n-1}|c_i(k(\tau+1)+\tau-s)|\right)H_i(l(\tau+1)+\tau-s)c^{l+\frac{r-s-1}{\tau+1}-n+1}\right],
   $$
   we have
   $$
   |x_i(m)-y_i(m)|\leq\left(\mathcal{A}+\frac{1}{1-\lambda}\mathcal{B}\right)c^{n-1}\|\ov{\alpha}-\ov{\beta}\|.
   $$
   From {\bf (H2)}, we obtain $\mathcal{B}\leq\lambda<1$ and $\mathcal{A}\leq1$, thus, recalling that $n=\frac{m+s}{\tau+1}$, $c\in]0,1]$ and $s\in\{0,\ldots,\tau\}$, we conclude   
   $$
   |x_i(m)-y_i(m)|\leq\frac{1}{1-\lambda}c^{n-1}\|\ov\alpha-\ov\beta\|\leq\frac{1}{1-\lambda}c^{\frac{m-(\tau+1)}{\tau+1}}\|\ov\alpha-\ov\beta\| .
   $$
   As $i$ is arbitrary, then condition \eqref{theo:hipotese-inducao} holds for $t=m$ and consequently \eqref{theo:hipotese-inducao} holds for all $m\in\en_0$.
   Finally, by \eqref{theo:hipotese-inducao}, we have
   $$
     \begin{array}{rcl}
     		\|\ov x_m-\ov y_m\|&=&\dst\max_{j=r,\ldots,0}\left\{|\ov x(m+j)-\ov y(m+j)|\right\}\\
     		&\leq&\dst\max_{j=r,\ldots,0}\left\{\frac{1}{1-\lambda}c^{\frac{m+j-(\tau+1)}{\tau+1}}\|\ov\alpha-\ov\beta\|\right\}\\
     		&=&\dst\frac{1}{1-\lambda}c^{\frac{m+r-(\tau+1)}{\tau+1}}\|\ov\alpha-\ov\beta\|
     		=\dst\dfrac{c^{\frac{r}{\tau+1}-1}}{1-\lbd} \,\left(c^{\frac{1}{\tau+1}}\right)^m\|\ov\alpha-\ov\beta\|
     \end{array}
   $$
\end{proof}

We remark that, Theorem \ref{thm:stability-criterion-discrete-model} assures that, under hypotheses (H1) and (H2), difference equation \eqref{discrete-general-model} is uniformly stable and, if $c<1$, then  \eqref{discrete-general-model} is globally exponentially stable.

\section{Periodic Model}\label{periodic-models}	
	
  In this section we assume that difference equation \eqref{discrete-general-model} is periodic, i.e., for some  $\omega\in\en$, we consider the difference equation 
  \begin{eqnarray}\label{discrete-general-model-periodico}
    x_i(m+1)
    = c_i(m)x_i(m-\tau)+h_i\left(m,\overline x_m\right),\Es m\in\en_0,\,i\in\{1,\ldots,N\},
  \end{eqnarray}
  with $c_i:\en_0\to]-1,1[$ and $h_i:\en_0\times X^N\to\er$ functions satisfying the hypotheses:
  \begin{description}
  	\item[(P1)] For all $i\in\{1,\ldots,N\}$ and  $m\in\en_0$, we have $c_i(m)=c_i(m+\omega)$;
  	\item[(P2)] for all $i\in\{1,\ldots,N\}$, $m\in\en_0$, and $\ov{\alpha}\in X^N$, we have $h_i(m+\omega,\ov{\alpha})=h_i(m,\ov{\alpha})$.
  \end{description}  
   \begin{theorem}\label{teo:exist-est-sol-periodic-model-geral}
   	Assume (P1), (P2), (H1), and (H2) with $c\in]0,1[$.
   	
   	Then there is a $\omega$-periodic solution $\ov x^*$ of  \eqref{discrete-general-model-periodico} verifying
   	\begin{eqnarray}\label{eq:estabilidade-periodica}
   		\|\ov x_m(\cdot,\ov{\alpha})-\ov x^*_m\|\leq \dfrac{c^{\frac{r}{\tau+1}-1}}{1-\lbd} \,\left(c^{\frac{1}{\tau+1}}\right)^m \|\ov\alpha-\ov x^*_0\|,
   	\end{eqnarray}
   	for all $m\in\en$.
   \end{theorem}
 
\begin{proof}
	From Theorem \ref{thm:stability-criterion-discrete-model}, inequality \eqref{eq:estabilidade-modelo-discreto1}, we have
	$$
     \|\ov x_m(\cdot,\ov{\alpha})-\ov x^*_m\|
	\le Q(m) \|\ov\alpha-\ov{\beta}\|,\Es\forall\ov{\alpha},\ov{\beta}\in X^N,\,\forall m\in\en_0,
	$$
	where $Q(m)=\dfrac{c^{\frac{r}{\tau+1}-1}}{1-\lbd} \,\left(c^{\frac{1}{\tau+1}}\right)^m$.	
	As $c<1$, then there is $p\in\en$ such that $Q(m)<1$ for all $m\in[p,+\infty[_\ze$. 
	
	Define the map $P:X^N\to X^N$ by $P(\ov\alpha)=\ov x_\omega(\cdot,\ov\alpha)$. For $\ov\alpha,\ov\beta\in X^N$, we have
	\begin{align*}
		\|P^p(\ov \alpha)-P^p(\ov \beta)\|
		& = \|P(P^{p-1}(\ov \alpha))-P(P^{p-1}(\ov \beta))\|\\
		& = \|\ov x_\omega(\cdot,P^{p-1}(\ov \alpha))-\ov x_\omega(\cdot,P^{p-1}(\ov \beta))\|\\
		& = \|\ov x_\omega(\cdot,\ov x_\omega(\cdot,P^{p-2}(\ov \alpha)))-\ov x_\omega(\cdot,\ov x_\omega(\cdot,P^{p-2}(\ov \beta)))\|,\\
	\end{align*}
	and from (P1) and (P2) the difference equation  \eqref{discrete-general-model-periodico} is $\omega$-periodic and consequently
	\begin{align*}
		\|P^p(\ov \alpha)-P^p(\ov \beta)\|
		& = \|\ov x_{2\omega}(\cdot,P^{p-2}(\ov \alpha))-\ov x_{2\omega}(\cdot,P^{p-2}(\ov \beta))\|\\
		& = \|\ov x_{p\omega}(\cdot,\ov \alpha)-\ov x_{p\omega}(\cdot,\ov \beta)\|\leq Q(p\omega)\|\ov \alpha-\ov \beta\|.\\
	\end{align*}
	As $Q(\omega p)<1$, then $P^p$ is a contraction map on Banach space $X^N$, thus there is a unique $\ov\alpha^*\in X^N$ such that $P^p(\ov \alpha^*)=\ov \alpha^*$. Consequently
	$$
	P^p(P(\ov \alpha^*))=P(P^p(\ov \alpha^*))=P(\ov \alpha^*),
	$$
	and we obtain $P(\ov \alpha^*)=\ov \alpha^*$, that is $\ov x_\omega(\cdot,\ov \alpha^*)=\ov \alpha^*$.
	
	As $\ov x(m,\ov \alpha^*)$ is a solution of the $\omega$-periodic difference equation \eqref{discrete-general-model-periodico}, then $\ov x(m+\omega,\ov \alpha^*)$ is also a solution of  \eqref{discrete-general-model-periodico} verifying
	$$
	\ov x(m,\ov \alpha^*)=\ov x(m,\ov x_\omega(\cdot,\ov \alpha^*))=\ov x(m+\omega,\ov \alpha^*),\ \ \ \forall m\in\en,
	$$
	which means that $\ov x^*(m)=\ov x(m,\ov \alpha^*)$ is a $\omega$-periodic solution of \eqref{discrete-general-model-periodico}.

	Finally, the inequality \eqref{eq:estabilidade-periodica} follows from Theorem \ref{thm:stability-criterion-discrete-model}.
\end{proof}

\section{Applications to Neural Network models}\label{applications}

In this section, we apply our main results to obtain criteria for the global exponential stability of several discrete-time neural network type models with delay in the leakage terms. Some criteria for the existence and global exponential stability of a periodic solution of periodic models are also established.

First of all, we consider the following discrete-time generalized neural network model with delay in the leakage terms
\begin{eqnarray}\label{eq:modelo-geral-aplicacao}
   x_i(m+1)=c_i(m)x_i(m-\tau)+\sum_{j=1}^Nh_{ij}\left(m,\ov{x}_m\right),\Es m\in\en_0,\,i=1,\ldots,N,
\end{eqnarray}
 where $N\in\en$, $\tau\in\en_0$, and $c_i:\en_0\to]-1,1[$ and $h_{ij}:\en_0\times X^N\to\er$  are functions such that the following hypothesis holds:
\begin{description}
	\item[(A1)] The functions $h_{ij}$ are Lipschitz, i.e. for each $i,j\in\{1,\ldots,N\}$  there exists a positive constant $H_{ij}$ such that
	$$
	|h_{ij}(m,\ov\alpha)-h_{ij}(m,\ov\beta)|\leq H_{ij}\|\ov\alpha-\ov\beta\|,\Es\forall m\in\en_0,\forall \ov\alpha,\ov\beta\in X^N.
	$$
\end{description}
 
Now, we apply Theorem \ref{thm:stability-criterion-discrete-model} to obtain the following stability result.
\begin{theorem}\label{teo:est-modelo-geral-aplicacoes-discreto}
  Assume (A1).
  
  If 
  \begin{eqnarray}\label{eq:cond-est-mod-geral-aplicacoes-disc}
  1-c_i^+>\sum_{j=1}^NH_{ij},\Es\forall i\in\{1,\ldots,,N\},
  \end{eqnarray}
  where $c_i^+ = \sup\limits_m |c_i(m)|$, then the model  \eqref{eq:modelo-geral-aplicacao} is globally exponentially stable.
  
\end{theorem}	
\begin{proof}
	 The model \eqref{eq:modelo-geral-aplicacao}  is a particular case of \eqref{discrete-general-model} with
	$$
	  h_i(m,\ov\alpha)=\sum_{j=1}^Nh_{ij}(m,\ov{\alpha}),
	$$
   for all $\ov{\alpha} \in X^N$, $m\in\en_0$, and $i,j\in\{1,\ldots,N\}$.
	
	From (A1), the hypothesis (H1) holds with $H_i(m)=\dst\sum_{j=1}^NH_{ij}$ for all $i\in\{1,\ldots,N\}$. 	
	
	
	Now we show that (H2) also holds.
	
	For each $i\in\{1,\ldots,N\}$ such that $c_i^+\neq0$, by \eqref{eq:cond-est-mod-geral-aplicacoes-disc}, we have $c_i^+\in]0,1[$ and define $\nu_i=-\ln(c_i^+)\in]0,+\infty[$.  Thus, we have $c_i^+=\e^{-\nu_i}$.
	
	For each $i\in\{1,\ldots,N\}$ such that $c_i^+=0$, by \eqref{eq:cond-est-mod-geral-aplicacoes-disc} it is possible to choose $\nu_i\in]0,+\infty[$ such that
	$$
	  1-\e^{-\nu_i}>\sum_{j=1}^NH_{ij}.
	$$  
	
	From \eqref{eq:cond-est-mod-geral-aplicacoes-disc}, we have
	$$
	  \frac{\e^{\nu_i}-1}{\e^{\nu_i}}>\sum_{j=1}^NH_{ij},\Es\forall i\in\{1,\ldots,,N\}.
	$$
	Consequently, there is a positive number $\mu<\dst\min_i\nu_i$ such that
	\begin{eqnarray}\label{eq:cond-est-mod-est-disc-modificada}
	    \frac{\e^{\nu_i-\mu}-1}{\e^{\nu_i}}\e^{\mu\frac{r}{\tau+1}}>\sum_{j=1}^NH_{ij},\Es\forall i\in\{1,\ldots,,N\}.
	\end{eqnarray}
	Defining $c=\e^{-\mu}$, we have $|c_i(m)|\leq c$ for all $m\in\en_0$. Consequently, for $i\in\{1,\ldots,N\}$ and $s\in\{0,\ldots,\tau\}$, we have
	
	\begin{align*}
	&\sup_{n\in\en}\left[\sum_{l=0}^{n-1}\left(\prod_{k=l+1}^{n-1}|c_i(k(\tau+1)+\tau-s)|\right)H_i(l(\tau+1)+\tau-s)c^{l+\frac{r-s-1}{\tau+1}-n+1}\right]\\
	& \leq\sup_{n\in\en}\left[\sum_{l=0}^{n-1}\e^{-\nu_i(n-1-l)}\e^{-\mu(l+\frac{r-s-1}{\tau+1}-n+1)}\sum_{j=1}^NH_{ij}\right]\\    
     &=\sup_{n\in\en}\left[\sum_{l=0}^{n-1}\e^{(\nu_i-\mu)\left(l-n\right)}\e^{\nu_i}\e^{-\mu\left(\frac{r-s-1}{\tau+1}+1\right)}\sum_{j=1}^NH_{ij}\right]\\ 
	&=\sup_{n\in\en}\left[\left(\sum_{l=0}^{n-1}\e^{(\nu_i-\mu)\left(l-n\right)}\right)\e^{\nu_i}\e^{-\mu\left(\frac{r-s-1}{\tau+1}+1\right)}\sum_{j=1}^NH_{ij}\right]\\	&=\sup_{n\in\en}\left[\left(\sum_{l=0}^{n-1}\e^{(\nu_i-\mu)l}\right)\e^{-(\nu_i-\mu)n}\e^{\nu_i}\e^{-\mu\left(\frac{r-s-1}{\tau+1}+1\right)}\sum_{j=1}^NH_{ij}\right]\\
	&=\sup_{n\in\en}\left[\frac{1-\e^{(\nu_i-\mu)n}}{1-\e^{\nu_i-\mu}}\e^{-(\nu_i-\mu)n}\e^{\nu_i}\e^{-\mu\left(\frac{r-s-1}{\tau+1}+1\right)}\sum_{j=1}^NH_{ij}\right]\\
	&=\sup_{n\in\en}\left[\frac{1-\e^{-(\nu_i-\mu)n}}{\e^{\nu_i-\mu}-1}\e^{\nu_i}\e^{-\mu\left(\frac{r-s-1}{\tau+1}+1\right)}\sum_{j=1}^NH_{ij}\right]\\
	&=\sup_{n\in\en}\left(1-\e^{-(\nu_i-\mu)n}\right)\e^{-\mu\left(\frac{r-s-1}{\tau+1}+1\right)}\frac{\e^{\nu_i}}{\e^{\nu_i-\mu}-1}\sum_{j=1}^NH_{ij}.\\
		\end{align*}
	As $\dst\sup_{n\in\en}\left(1-\e^{-(\nu_i-\mu)n}\right)=1$ and $s\in\{1,\ldots,\tau\}$, we have
	\begin{align*}
    \lambda&\leq \max_{i=1,\ldots,N}\left(\e^{-\mu\frac{r}{\tau+1}}\frac{\e^{\nu_i}}{\e^{\nu_i-\mu}-1}\sum_{j=1}^NH_{ij}\right)\\
	\end{align*}
	and, by \eqref{eq:cond-est-mod-est-disc-modificada}, we obtain
	\begin{align*}
	\lambda&<\max_{i=1,\ldots,N}\left(\e^{-\mu\frac{r}{\tau+1}}\frac{\e^{\nu_i}}{\e^{\nu_i-\mu}-1}\frac{\e^{\nu_i-\mu}-1}{\e^{\nu_i}}\e^{\mu\frac{r}{\tau+1}}\right)=1,\\ 
	\end{align*}
    thus hypotheses (H2) holds. From Theorem~\ref{thm:stability-criterion-discrete-model}, we obtain that  model  \eqref{eq:modelo-geral-aplicacao} is globally exponentially stable.
\end{proof}

If \eqref{eq:modelo-geral-aplicacao} is a periodic model then, from Theorem \ref{teo:exist-est-sol-periodic-model-geral}, we establish sufficient conditions for the existence and global exponential stability of a periodic solution.

\begin{theorem}\label{teo:modelo-geral-aplicacao-periodico}
	Assume (A1) and there is $\omega\in\en$ such that
	$$
	c_i(m)=c_i(m+\omega) \ \ \ \text{and} \ \ \ h_{ij}(m+\omega,\ov\alpha)=h_{ij}(m,\ov\alpha)
	$$
	for all $m\in\en_0$, $\ov\alpha\in X^N$, and $i,j\in\{1,\ldots,N\}$.
	
	If condition \eqref{eq:cond-est-mod-geral-aplicacoes-disc} holds, then there is a $\omega$-periodic solution of\eqref{eq:modelo-geral-aplicacao} which is globally exponentially stable. 
\end{theorem}

\begin{proof}
	By the hypotheses, conditions (P1) and (P2) hold. From the proof of Theorem \ref{teo:est-modelo-geral-aplicacoes-discreto}, the conditions (H1) and (H2) also hold with $c\in]0,1[$. Thus, the result follows from Theorem \ref{teo:exist-est-sol-periodic-model-geral}.
\end{proof}

As a particular case of \eqref{eq:modelo-geral-aplicacao} we consider the following discrete-time low-order Hopfield neural network model with delay in leakage term
\begin{eqnarray}\label{eq:modelo-Hopfield-geral}
	x_i(m+1)=c_i(m)x_i(m-\tau)+\sum_{j=1}^N\sum_{k=1}^K b_{ijk}(m)f_{ijk}\left(x_j(m-\tau_{ijk}(m))\right)+I_i(m),
\end{eqnarray}
for $m \in \N_0$, $i\in\{1,\ldots,N\}$, where $N,K\in\en$, $\tau\in\en_0$, and $c_i:\en_0\to]-1,1[$, $\tau_{ijk}:\en_0\to\en_0$, $b_{ijk},I_i:\en_0\to\er$, and $f_{ijk}:\er\to\er$  are functions such that the following hypotheses hold:
\begin{description}
	\item[(B1)] The functions $b_{ijk}$ and $\tau_{ijk}$ are bounded;
	\item[(B2)] the functions $f_{ijk}$ are Lipschitz, i.e. for each $i,j\in\{1,\ldots,N\}$ and $k\in\{1,\ldots,K\}$, there exists a constant $F_{ijk}$ such that
	$$
		|f_{ijk}(x)-f_{ijk}(y)|\leq F_{ijk}|x-y|,\Es\forall x,y\in\er.
	$$
\end{description}
The discrete-time autonomous Hopfield neural network model, studied in \cite{Hong_Ma-MBE-2019},
\begin{eqnarray}\label{eq:modelo-Hopfield-hong+ma}
	x_i(m+1)=c_ix_i(m)+\sum_{j=1}^N b_{ij}f_j\left(x_j(m-\tau_{ij})\right),\Es m\in\en_0,\,i=1,\ldots,N,
\end{eqnarray}
and the discrete-time Hopfield neural network model with delay in leakage term and constant coefficients, studied in \cite{Qui_Liu_Shu-ADE-2016},
\begin{eqnarray}\label{eq:modelo-Hopfield-qiu+liu+shu}
	x_i(m+1)=c_ix_i(m-\tau)+\sum_{j=1}^N a_{ij}f_j\left(x_j(m)\right)+\sum_{j=1}^N b_{ij}f_j\left(x_j(m-\sigma(m))\right)+I_i ,\Es
\end{eqnarray}
for $m\in\en_0$, $i=1,\ldots,N$, are particular case of model \eqref{eq:modelo-Hopfield-geral}. 

For $i,j\in\{1,\ldots,N\}$ and $k\in\{1,\ldots,K\}$, in what follows we use the  notation
$$
r=-\max_{i,j,k,m}\{\tau_{ijk}(m),\tau\}, \ \ \ c_i^+ = \sup\limits_m |c_i(m)|, \ \ \ \text{ and }\ \ \ b_{ijk}^+ = \sup\limits_m |b_{ijk}(m)|.
$$
From Theorem \ref{teo:est-modelo-geral-aplicacoes-discreto}, we obtain the global exponential stability of \eqref{eq:modelo-Hopfield-geral}.
\begin{theorem}\label{teo:est-modelo-hopfiel-geral-discreto}
	Assume (B1) and (B2).
	
	If 
	\begin{eqnarray*}\label{eq:cond-est-mod-hopfield-geral-disc}
		1-c_i^+>\sum_{j=1}^N\sum_{k=1}^Kb_{ijk}^+F_{ijk},\Es\forall i\in\{1,\ldots,,N\},
	\end{eqnarray*}
	then the model  \eqref{eq:modelo-Hopfield-geral} is globally exponentially stable.
\end{theorem}	

\begin{proof}
	The model \eqref{eq:modelo-Hopfield-geral}  is a particular case of model \eqref{eq:modelo-geral-aplicacao} with
	$$
	h_{ij}(m,\ov\alpha)=\left(\sum_{k=1}^Kb_{ijk}(m)f_{ijk}(\alpha_j(-\tau_{ijk}(m)))\right)+\frac{I_i(m)}{N},
	$$
	for all $\ov{\alpha}=\prts{\alpha_1, \ldots, \alpha_N} \in X^N$, $m\in\en_0$, and $i,j\in\{1,\ldots,N\}$.
	
 By (B1) and (B2), the hypothesis (A1) holds with $H_{ij}=\dst\sum_{k=1}^Kb_{ijk}^+F_{ijk}$, for all $i,j\in\{1,\ldots,N\}$, thus the conclusion follows from Theorem \ref{teo:est-modelo-geral-aplicacoes-discreto}. 	
\end{proof}

The previous result is improved in the following result.

Consider the $N\times N$-matrix $\mathcal{M}$ defined by
\begin{eqnarray*}\label{matriz-M}
\label{M-matriz}
\mathcal{M}=\mathcal{I}-diag(c_1^+,\ldots,c_N^+)-\left[\left(\sum_{k=1}^Kb^+_{ijk}F_{ijk}\right)_{ij}\right],
\end{eqnarray*}
where $\mathcal{I}$ is the  $N\times N$ identity matrix.
\begin{corollary}\label{cor:aplicacao2-est-modelo-hopfield-discreto}
	 Assume (B1) and (B2).
	 
  If $\mathcal{M}$ is a non-singular M-matrix, then the model  \eqref{eq:modelo-Hopfield-geral} is globally exponentially stable.
\end{corollary}

\begin{proof}
	If $\mathcal{M}$ is a non-singular M-matrix, then (see \cite{fiedler}) there is $\ov d=(d_1,\ldots,d_N)>0$ such that $\mathcal{M}\ov d > 0$, i.e.,
	\begin{eqnarray}\label{eq:a_i^->alpha...2}
		d_i(1-c_i^+) > \sum_{j=1}^{N} d_j\left(\sum_{k=1}^K b_{ijk}^+F_{ijk}\right),\Es\forall i\in\{1,\ldots,N\}.
	\end{eqnarray}
	The change $y_i(m)=d_i^{-1}x_i(m)$, $m\in\en$ and $i\in\{1,\ldots,N\}$, transforms \eqref{eq:modelo-Hopfield-geral} into
	\begin{eqnarray*} 
		y_i(m+1)=c_i(m)y_i(m-\tau)+\sum_{j=1}^N\sum_{k=1}^K \tilde{b}_{ijk}(m)\tilde{f}_{ijk}\left(y_j(m-\tau_{ijk}(m))\right),
	\end{eqnarray*}
	where
	$$
	\tilde{b}_{ijk}(m)=d_i^{-1}b_{ijk}(m) \ \ \text{and}\ \ \tilde{f}_{ijk}(u)=f_{ijk}(d_ju), 
	$$
	for $m\in\en_0$ and $u\in\er$. As $f_{ijk}$ are Lipschitz functions with constant $F_{ijk}$, then $\tilde{f}_{ijk}$ verify (B2) with constant $\tilde{F}_{ijk}=d_jF_{ijk}$. From \eqref{eq:a_i^->alpha...2} we have
	$$
	1-c_i^+> \sum_{j=1}^{N}\sum_{k=1}^K \left(d_i^{-1}b_{ijk}^+\right)\left(d_jF_{ijk}\right),\Es\forall i\in\{1,\ldots,N\}.
	$$
	which is equivalent to
	$$
	1-c_i^+ > \sum_{j=1}^{N}\sum_{k=1}^K\tilde{b}_{ijk}^+\tilde{F}_{ijk},\Es\forall i\in\{1,\ldots,N\}.
	$$
	and the result follows from the Theorem \ref{teo:est-modelo-hopfiel-geral-discreto}.
\end{proof}

Now we consider the model \eqref{eq:modelo-Hopfield-geral} with periodic coefficients, i.e., we assume that there is $\omega\in\en$ such that $c_i(m)$, $b_{ijk}(m)$, $\tau_{ijk}(m)$ and $I_i(m)$ are $\omega$-periodic functions. Naturally we have
 \begin{align*}
   c_i^+ = \max\set{|c_i(1)|,\ldots,|c_i(\omega)|} \ \ \ \text{and} \ \ \
 b_{ijk}^+ = \max\set{|b_{ijk}(1)|,\ldots,|b_{ijk}(\omega)|},
 \end{align*}
 for $i,j\in\{1,\ldots,N\}$, $k\in\{1,\ldots,K\}$.
 
 From Theorem \ref{teo:modelo-geral-aplicacao-periodico} and Corollary \ref{cor:aplicacao2-est-modelo-hopfield-discreto} we have the following result, which extends, to models with delay in the leakage terms, the result \cite[Theorem 4]{Bento-oliveira-Silva}.

 \begin{corollary}\label{cor:exist-est-sol-periodic-model-hopfiels-periodico}
 	Assume $c_i$, $b_{ijk}$, $\tau_{ijk}$, and $I_i$ are $\omega$-periodic function, and (B2).
 	
 	If the matrix $\mathcal{M}$, defined by \eqref{matriz-M}, is a non-singular M-matrix, then model \eqref{eq:modelo-Hopfield-geral} has a unique $\omega$-periodic solution which is globally exponentially stable. 
 \end{corollary}

Considering model  \eqref{eq:modelo-Hopfield-geral} with autonomous coefficients that is, for each $i,j\in\{1,\ldots,N\}$ and $k\in\{1,\ldots,K\}$ we have 
$$
  c_i(m)=c_i, \ \ \ I_i(m)=I_i, \ \ \ \text{and} \ \ \ b_{ijk}(m)=b_{ijk}, \Es\forall m\in\en_0, 
$$
with $c_i\in]-1,1[$ and $b_{ijk},I_i\in\er$, the previous Corollary \ref{cor:exist-est-sol-periodic-model-hopfiels-periodico} allows us to obtain the following result.

 \begin{corollary}\label{cor:exist-est-equilibrio-model-hopfiels-autonomo}
	Assume (B2).
	
	If the matrix 
	\begin{eqnarray}\label{matriz-N}
	\mathcal{N}=\mathcal{I}-diag(|c_1|,\ldots,|c_N|)-\left[\left(\sum_{k=1}^K|b_{ijk}|F_{ijk}\right)_{ij}\right]
	\end{eqnarray}
   is a non-singular M-matrix, then the model
   \begin{eqnarray}\label{eq:modelo-Hopfield-geral-autonomo}
   	x_i(m+1)=c_ix_i(m-\tau)+\sum_{j=1}^N\sum_{k=1}^K b_{ijk}f_{ijk}\left(x_j(m-\tau_{ijk}(m))\right)+I_i
   \end{eqnarray}
for $m \in \N_0$, $i\in\{1,\ldots,N\}$, has a unique equilibrium which is globally exponentially stable. 
\end{corollary}
 
In \cite[Theorem 3.1]{Hong_Ma-MBE-2019}, the global attractivity  of \eqref{eq:modelo-Hopfield-hong+ma}, a particular case of \eqref{eq:modelo-Hopfield-geral-autonomo}, is obtained assuming that activation functions $f_j:\er\to\er$ are differentiable such that 
$$
  f_j(0)=0, \  \lim_{x\to+\infty}f_j(x)=1, \ \lim_{x\to-\infty}f_j(x)=-1,
$$
and
$$
 \ f'_j(0)=\sup_xf'_j(x)=1, \ f'_j(u)>0,\forall u\in\er,
$$
 joint with the condition of 
 $$
   \mathcal{I}-diag(|c_1|,\ldots,|c_n|)-[|b_{ij}|]
 $$
 being an M-matrix. By hypotheses in \cite{Hong_Ma-MBE-2019}, it is clear that $f_j$ are Lipschitz functions with Lipschitz constant equal to 1. We remark that Corollary \ref{cor:exist-est-equilibrio-model-hopfiels-autonomo} states the global exponential stability of  \eqref{eq:modelo-Hopfield-hong+ma}, instead of the global attractivity as in \cite[Theorem 3.1]{Hong_Ma-MBE-2019}, but we assume that $\mathcal{N}$, defined by \eqref{matriz-N}, is a non-singular M-matrix which is more restrictive than $\mathcal{N}$ just be an M-matrix as is assumed in  \cite[Theorem 3.1]{Hong_Ma-MBE-2019}.\\
 
 Model \eqref{eq:modelo-Hopfield-qiu+liu+shu}, studied in \cite{Qui_Liu_Shu-ADE-2016}, is also a particular case of \eqref{eq:modelo-Hopfield-geral-autonomo}, but the asymptotic stability of \eqref{eq:modelo-Hopfield-qiu+liu+shu} is established in \cite{Qui_Liu_Shu-ADE-2016} under a different hypotheses set.\\
 
It is relevant to observe that model \eqref{eq:modelo-geral-aplicacao} is general enough to include, as particular cases, some BAM neural network models with delay in leakage term.     
 
 Considering, in the general model \eqref{eq:modelo-geral-aplicacao}, $N=N_1+N_2$, for $N_1,N_2\in\en$,
 $$
 c_i(m)=\left\{\begin{array}{ll}
 	\hat{c}_i(m),&i=1,\ldots,N_1\\
 	\til{c}_{i-N_1}(m),&i=N_1+1,\ldots,N_1+N_2
 \end{array}\right.,\Es\forall m\in\en_0,
 $$
 and
 $$
  h_{ij}(m,\ov\alpha)=\left\{\begin{array}{ll}
   	0,&i=1,\ldots,N_1,\,j=1,\ldots,N_1\\
   	\begin{array}{l}
   		\hat{a}_{i(j-N_1)}(m)f_{j-N_1}(\alpha_j(0))\\
   		+\hat{b}_{i(j-N_1)}(m)f_{j-N_1}(\alpha_j(-\hat\tau_{i(j-N_1)}(m)))\\
   		+\frac{\hat{I}_i(m)}{N_2},
   	\end{array}&i=1,\ldots,N_1,\,j=N_1+1,\ldots,N_1+N_2\\
   	\begin{array}{l}
   		\til{a}_{(i-N_1)j}(m)g_j(\alpha_j(0))\\
   		+\til{b}_{(i-N_1)j}(m)g_j(\alpha_j(-\til\tau_{(i-N_1)j}(m)))\\
   		+\frac{\til{I}_i(m)}{N_1},
   	\end{array}&i=N_1+1,\ldots,N_1+N_2,\,j=1,\ldots,N_1\\
   	0,&\begin{array}{l}
   		i=N_1+1,\ldots,N_1+N_2,\\j=N_1+1,\ldots,N_1+N_2
   	\end{array}
   \end{array}\right.,
 $$ 
 for all $m\in\en_0$ and $\ov\alpha\in X^{N_1+N_2}$,
 we have the following BAM neural network model
\begin{eqnarray}\label{eq:BAM}
	\left\{\begin{split}
		x_i(m+1)=\lefteqn{\hat{c}_i(m)x_i(m-\tau)+\dst\sum_{j=1}^{N_2}\hat{a}_{ij}(m)f_j(y_j(m))}\\
		& +\dst\sum_{j=1}^{N_2}\hat{b}_{ij}(m)f_j(y_j(m-\hat\tau_{ij}(m)))+\hat{I}_i(m),\Es i=1,\ldots,N_1,\\
		y_j(m+1)=\lefteqn{\til{c}_j(m)y_j(m-\tau)+\dst\sum_{i=1}^{N_1}\widetilde{a}_{ji}(m)g_i(x_i(m))}\\
		& +\dst\sum_{i=1}^{N_1}\til{b}_{ji}(m)g_i(x_i(m-\til\tau_{ji}(m)))+\til{I}_j(m),\Es
		j=1,\ldots,N_2,\\
	\end{split}\right.
\end{eqnarray}
where $\hat{c}_i,\til{c}_j:\en_0\to]-1,1[$, $\hat{a}_{ij},\widetilde{a}_{ji}:\en_0\to\er$,  $\hat{\tau}_{ij},\til{\tau}_{ji}:\en_0\to\en_0$ are bounded functions, $\hat{I}_i,\til{I}_j:\en_0\to\er$ are functions, and $f_j,g_i:\er\to\er$ are Lipschitz functions with Lipschitz constants $F_j$ and $G_i$ respectively.

For the functions in the model \eqref{eq:BAM}, we use the notation
$$
r=-\max_{i,j,m}\{\hat\tau_{ij}(m),\til\tau_{ji}(m),\tau\}, \ \ \ \hat{c}_i^+ = \sup\limits_m |\hat{c}_i(m)|, \ \ \  \til{c}_j^+ = \sup\limits_m |\til{c}_j(m)|,
$$
$$
\hat{a}_{ij}^+ = \sup\limits_m |\hat{a}_{ij}(m)|, \ \ \hat{b}_{ij}^+ = \sup\limits_m |\hat{b}_{ij}(m)|, \ \ \til{a}_{ji}^+ = \sup\limits_m |\til{a}_{ji}(m)|, \ \text{ and } \ \til{b}_{ji}^+ = \sup\limits_m |\til{b}_{ji}(m)|.
$$
We also define the matrix $\mathcal{P}$ by
$$
  \mathcal{P}:=\left[\begin{array}{cc}
  	\mathcal{I}_{N_1}-\hat{C} & -U\\
  	-S&\mathcal{I}_{N_2}-\til{C}
  	\end{array}\right],
$$
where, for $k=1,2$, $\mathcal{I}_{N_k}$ is the $N_k\times N_k$ identity matrix, $\hat{C}=diag(\hat{c}_1^+,\ldots,\hat{c}_{N_1}^+)$, $\til{C}=diag(\til{c}_1^+,\ldots,\til{c}_{N_2}^+)$, $U=\big[(\hat{a}_{ij}^++\hat{b}_{ij}^+)F_j\big]$, and $S=\big[(\til{a}_{ji}^++\til{b}_{ji}^+)G_i\big]$.
Following the same arguments presented in the proofs of Theorem \ref{teo:est-modelo-hopfiel-geral-discreto} and Corollary \ref{cor:aplicacao2-est-modelo-hopfield-discreto}, we obtain the next exponential stability criterion.

\begin{theorem}\label{teo:criterio-estabilidade-BAM}
  If $\mathcal{P}$ is a non-singular M-matrix, then the model \eqref{eq:BAM} is globally exponentially stable. 	
\end{theorem}

From Theorem \ref{teo:modelo-geral-aplicacao-periodico} and Theorem \ref{teo:criterio-estabilidade-BAM}, we obtain the following stability results for model \eqref{eq:BAM} with periodic and constant coefficients, respectively.  
 \begin{corollary}
	Assume $\hat{c}_i$, $\til{c}_j$ $\hat{a}_{ij}$, $\til{a}_{ji}$, $\hat{b}_{ij}$, $\til{b}_{ji}$, $\hat{\tau}_{ij}$, $\til{\tau}_{ji}$, $\hat{I}_i$, and $\til{I}_j$ are $\omega$-periodic functions.
	
	If $\mathcal{P}$ is a non-singular M-matrix, then the model \eqref{eq:BAM} has a unique $\omega$-periodic solution which is globally exponentially stable. 
\end{corollary}

\begin{corollary}\label{cor:exist-est-sol-periodic-model-BAM-autonomo}
	Assume $\hat{c}_i(m)=\hat{c}_i$, $\til{c}_j(m)=\til{c}_j$, $\hat{a}_{ij}(m)=\hat{a}_{ij}$, $\til{a}_{ji}(m)=\til{a}_{ji}$, $\hat{b}_{ij}(m)=\hat{b}_{ij}$, $\til{b}_{ji}(m)=\til{b}_{ji}$, $\hat{I}_i(m)=\hat{I}_i$, $\til{I}_j(m)=\til{I}_j$ for all $m\in\en_0$.
	
	If 
	$$
	\mathcal{I}-diag(|\hat{c}_1|,\ldots,|\hat{c}_{N_1}|,|\til{c}_1|,\ldots,|\til{c}_{N_2}|)-\left[\begin{array}{cc}
		0&\big((|\hat{a}_{ij}|+|\hat{b}_{ij}|)F_j\big)_{ij}\\
		\big((|\til{a}_{ji}|+|\til{b}_{ji}|)G_i\big)_{ji}&0
		\end{array}\right]
	$$ is a non-singular M-matrix, then the model \eqref{eq:BAM} has a unique equilibrium which is globally exponentially stable. 
\end{corollary}

The global exponential stability of \eqref{eq:BAM}, with constant coefficients but without delay in the leakage terms, $\tau=0$, also was established in \cite{Raja_Anthoni-CNSNS-2011} but with a different hypotheses set. Also with a different hypotheses set, the global asymptotic stability of \eqref{eq:BAM} with constant coefficients and different delays in the leakage terms, was established in \cite{Sowmiya_Raja_Cao_Rajchakit_Alsaedi-AJC-2020}.\\

Now we consider the following discrete-time high-order Hopfield neural network model with delay in leakage terms 

\begin{eqnarray}\label{eq:hopfield-high-order}
		x_i(m+1)=\lefteqn{c_i(m)x_i(m-\tau)+\dst\sum_{j=1}^{N}a_{ij}(m)f_j(x_j(m))}\nonumber\\
	& +\dst\sum_{j=1}^{N}\sum_{l=1}^Nb_{ijl}(m)g_j(x_j(m-\tau_{ijl}(m)))g_l(x_l(m-\xi_{ijl}(m))),
\end{eqnarray}
for $m\in\en_0$, $i\in\{1,\ldots,N\}$, where $N\in\en$, $\tau\in\en_0$, and $c_i:\en_0\to]-1,1[$, $\tau_{ijl},\xi_{ijl}:\en_0\to\en_0$, $a_{ij},b_{ijl}:\en_0\to\er$, and $f_j,g_j:\er\to\er$  are functions such that the following hypotheses hold:
\begin{description}
	\item[(HO1)] The functions $a_{ij},b_{ijl}$, and $\tau_{ijl},\xi_{ijl}$ are bounded, and consider
	$$
	  r=-\max_{i,j,l,m}\{\tau_{ijl}(m),\xi_{ijl}(m),\tau\}, \ \ \ c_i^+ = \sup\limits_m |c_i(m)|,
	  $$
	  $$
	  a_{ij}^+ = \sup\limits_m |a_{ij}(m)|, \ \ b_{ijl}^+ = \sup\limits_m |b_{ijl}(m)|;
	$$
	\item[(HO2)] for each $j\in\{1,\ldots,N\}$, the functions $f_j$ and $g_j$ are Lipschitz i.e.,  there exist constants $F_j$ and $G_j$ such that
	$$
	|f_j(x)-f_j(y)|\leq F_j|x-y|\text{ and }|g_j(x)-g_j(y)|\leq G_j|x-y|,\,\,\,\forall x,y\in\er;
	$$
	\item[(HO3)] for each $j\in\{1,\ldots,N\}$, the function $g_j$ is bounded i.e., there exists $m_j>0$ such that
	$$
	  |g_j(u)|\leq m_j,\Es\forall u\in\er.
	$$
\end{description}
From Theorem \ref{teo:est-modelo-geral-aplicacoes-discreto} we obtain the following global exponential stability result.
\begin{theorem}\label{teo:criterio-estabilidade-HOHM}
	Assume (HO1), (HO2), and (HO3).
	
	If there is $\ov q=(q_1,\ldots,q_N)>0$ such that 
	\begin{eqnarray}\label{eq:h-o-model-hipotese}
		d_i(1-c_i^+) > \sum_{j=1}^{N} \left(d_jF_ja_{ij}^++\sum_{l=1}^N b_{ijl}^+\left(m_jd_lG_l+m_ld_jG_j\right)\right),\,\,\forall i\in\{1,\ldots,N\},
	\end{eqnarray}
	then model \eqref{eq:hopfield-high-order} is globally exponentially stable.
\end{theorem}
\begin{proof}
	The change $y_i(m)=d_i^{-1}x_i(m)$, $m\in\en$ and $i\in\{1,\ldots,N\}$, transforms \eqref{eq:hopfield-high-order} into
	\begin{eqnarray}\label{eq:hopfield-high-order-proof}
		y_i(m+1)=\lefteqn{c_i(m)y_i(m-\tau)+\dst\sum_{j=1}^{N}\til{a}_{ij}(m)\til{f}_j(y_j(m))}\nonumber\\
		& +\dst\sum_{j=1}^{N}\sum_{l=1}^N\til{b}_{ijl}(m)\til{g}_j(y_j(m-\tau_{ijl}(m)))\til{g}_l(y_l(m-\xi_{ijl}(m))),
	\end{eqnarray}
	where
	$$
	\til{a}_{ij}(m)=d_i^{-1}a_{ij}(m), \ \ \tilde{b}_{ijl}(m)=d_i^{-1}b_{ijl}(m),\ \ \tilde{f}_{j}(u)=f_{j}(d_ju), \ \ \text{and}\ \ \tilde{g}_{j}(u)=g_{j}(d_ju), 
	$$
	for all $m\in\en_0$, $i,j,l\in\{1,\ldots,N\}$, and $u\in\er$. Model \eqref{eq:hopfield-high-order-proof} is a particular case of model \eqref{eq:modelo-geral-aplicacao} with 
\begin{eqnarray}\label{eq:def-h-high-ordel-model}
h_{ij}(m,\ov\alpha)=\til{a}_{ij}(m)\til{f}_j(\alpha_j(0))+\sum_{l=1}^N\til{b}_{ijl}(m)\til{g}_j(\alpha_j(-\tau_{ijl}(m)))\til{g}_l(\alpha_l(-\xi_{ijl}(m))),
\end{eqnarray}
for all $\ov{\alpha}=\prts{\alpha_1, \ldots, \alpha_N} \in X^N$, $m\in\en_0$, and $i,j\in\{1,\ldots,N\}$.

For each $i,j\in\{1,\ldots,N\}$, from (HO2) and (HO3) the function $h_{ij}$, defined by \eqref{eq:def-h-high-ordel-model}, verifies 
$$
  \begin{array}{l}
  	|h_{ij}(m,\ov\alpha)-h_{ij}(m,\ov\beta)|
  	\leq d_i^{-1}a_{ij}^+|\til{f}_j(\alpha_j(0))-\til{f}_j(\beta_j(0))|\\
  	+\dst\sum_{l=1}^Nd_i^{-1}b_{ijl}^+|\til{g}_j(\alpha_j(-\tau_{ijl}(m)))\til{g}_l(\alpha_l(-\xi_{ijl}(m)))
  	-\til{g}_j(\beta_j(-\tau_{ijl}(m)))\til{g}_l(\beta_l(-\xi_{ijl}(m)))|\\
  	\leq d_i^{-1}a_{ij}^+F_jd_j|\alpha_j(0)-\beta_j(0)|\\
  	
  	+\dst\sum_{l=1}^Nd_i^{-1}b_{ijl}^+|\til{g}_j(\alpha_j(-\tau_{ijl}(m)))\til{g}_l(\alpha_l(-\xi_{ijl}(m)))-\til{g}_j(\alpha_j(-\tau_{ijl}(m)))\til{g}_l(\beta_l(-\xi_{ijl}(m)))|\\
  	+\dst\sum_{l=1}^Nd_i^{-1}b_{ijl}^+|\til{g}_j(\alpha_j(-\tau_{ijl}(m)))\til{g}_l(\beta_l(-\xi_{ijl}(m)))-\til{g}_j(\beta_j(-\tau_{ijl}(m)))\til{g}_l(\beta_l(-\xi_{ijl}(m)))|\\
  	\leq d_i^{-1}a_{ij}^+F_jd_j\|\ov\alpha-\ov\beta\|
  	+\dst\sum_{l=1}^Nd_i^{-1}b_{ijl}^+m_jG_ld_l|\alpha_l(-\xi_{ijl}(m))-\beta_l(-\xi_{ijl}(m))|\\
  	+\dst\sum_{l=1}^Nd_i^{-1}b_{ijl}^+m_lG_jd_j|\alpha_j(-\tau_{ijl}(m))-\beta_j(-\tau_{ijl}(m))|\\
  	\leq \left(d_i^{-1}a_{ij}^+d_jF_j+\dst\sum_{l=1}^Nd_i^{-1}b_{ijl}^+\left(m_jd_lG_l+m_ld_jG_j\right)\right)\|\ov\alpha-\ov\beta\|,\\
  \end{array}
$$
for all $m\in\en_0$, and $\ov\alpha,\ov\beta\in X^N$. Consequently hypothesis (A1) holds with 
$$
  H_{ij}=\left(d_i^{-1}a_{ij}^+d_jF_j+\dst\sum_{l=1}^Nd_i^{-1}b_{ijl}^+\left(m_jd_lG_l+m_ld_jG_j\right)\right),\Es\forall i,j\in\{1,\ldots,N\}.
$$
By hypothesis \eqref{eq:h-o-model-hipotese}, condition \eqref{eq:cond-est-mod-geral-aplicacoes-disc} also holds and the result follows from Theorem \ref{teo:est-modelo-geral-aplicacoes-discreto}.
\end{proof}
Now we consider model \eqref{eq:hopfield-high-order} with periodic delays and coefficients functions.  From Theorem \ref{teo:modelo-geral-aplicacao-periodico} and the proof of Theorem \ref{teo:criterio-estabilidade-HOHM}, we obtain the next result.
	\begin{corollary}\label{cor:exist-est-sol-periodic-model-HOHM-periodico}
		Assume $c_i$, $a_{ij}$, $b_{ijl}$, $\tau_{ijl}$, and $\xi_{ijl}$ are $\omega$-periodic functions.
		
		If there is $\ov q=(q_1,\ldots,q_N)>0$ such that condition \eqref{eq:h-o-model-hipotese} holds, then \eqref{eq:hopfield-high-order} has a unique $\omega$-periodic solution which is globally exponentially stable. 
	\end{corollary}
	
 The exponential stability of  \eqref{eq:hopfield-high-order} with constants coefficients, $\tau_{ijl}(m)=\xi_{ijl}(m)$ for all $i,j,l\in\{1,\ldots,N\}$ and $m\in\en_0$, and without delay in the leakage terms ($\tau=0$) was recently studied in \cite{Dong_Wang_Zhang-AMC-2020}. The authors also assume that
 $$
   f_i(0)=g_i(0)=0,\Es\forall i\in\{1,\ldots,N\},
 $$ 
 which implies that $x=0$ is an equilibrium point of 
 \begin{eqnarray}\label{eq:hopfield-high-order-constant-coefficients}
 	x_i(m+1)=\lefteqn{c_ix_i(m)+\dst\sum_{j=1}^{N}a_{ij}f_j(x_j(m))}\nonumber\\
 	& +\dst\sum_{j=1}^{N}\sum_{l=1}^Nb_{ijl}g_j(x_j(m-\tau_{ijl}(m)))g_l(x_l(m-\tau_{ijl}(m))),
 \end{eqnarray}
where $c_i\in]-1,1[$ and $a_{ij},b_{ijl}\in\er$.
Under all these restrictions, in \cite{Dong_Wang_Zhang-AMC-2020} the global exponential stability of the zero solution of  \eqref{eq:hopfield-high-order-constant-coefficients} is obtained with the hypothesis: there is $\ov q=(q_1,\ldots,N)>0$ such that
\begin{eqnarray*}\label{eq:h-o-model-hipotese-coeficients-constantes}
	d_i(1-|c_i|) > \sum_{j=1}^{N} \left(d_jF_j|a_{ij}|+\sum_{l=1}^N |b_{ijl}|m_jd_lG_l\right),\,\,\forall i\in\{1,\ldots,N\},
\end{eqnarray*}
which is a slight weaker condition than \eqref{eq:h-o-model-hipotese}.

\section{Numerical simulation}\label{numericalexample}
In this section, we give a numerical example to illustrate the effectiveness of some the new results presented in this paper.

In model  \eqref{eq:modelo-Hopfield-geral} with $N=2$, let $K=2$ and
\begin{align*}
	&
	c_1(m) = \frac{1}{4}\cos \dfrac{2m\pi}{\omega}, \ \ \ \
	c_2(m) = \frac{1}{12}\sin\dfrac{2m\pi}{\omega}, \ \ \ \
	b_{111}(m) = \frac{1}{8}\cos\dfrac{2m\pi}{\omega}\\
	&
	b_{112}(m) =\frac{1}{8}\sin\dfrac{2m\pi}{\omega}, \ \ \ \
	b_{121}(m) =0, \ \ \ \
	b_{221}(m) =-\frac{1}{6}\sin\dfrac{2m\pi}{\omega}, \ \ \ \
	b_{122}(m) =\frac{1}{6}\sin\dfrac{2m\pi}{\omega}\\
	&
	b_{211}(m) = \frac{1}{4}\cos\dfrac{2m\pi}{\omega}, \ \ \ \
	b_{212}(m) = \frac{1}{4}\sin\dfrac{2m\pi}{\omega}, \ \ \ \
	b_{222}(m) = -\frac{5}{12}\sin\dfrac{2m\pi}{\omega},\\
	&
	\tau_{111}(m) = \tau_{121}(u) = \tau_{211}(m) = \tau_{221}(m) =0,\\
	&
	\tau_{112}(m) = \tau_{122}(u) = \tau_{212}(m) = \tau_{222}(m) =2+(-1)^m,\\
	&
	f_{111}(u) = f_{121}(u) = f_{211}(u) = f_{221}(u) =\arctan u, \ \ \ \ I_1(m)=0, \ \ \ \ I_2(m)=\frac{1}{2}\cos\frac{2m\pi}{\omega},\\
	&
	f_{112}(u) = f_{122}(u) = f_{212}(u) = f_{222}(u) = \tanh u, \ \ \ \
	\omega = 10, \ \ \ \ \tau=2, \ \ \ \ r=3,
\end{align*}
thus all coefficients and delay functions are 10-periodic.

We have $F_{ijk}=1$ for all $i,j,k\in\{1,2\}$ and the M-matrix, defined by \eqref{matriz-M}, has the form 
$$
  \mathcal{M}=\left[\begin{array}{cc}
  	1 &0\\
  	0&1\\
  	\end{array}\right]-\left[\begin{array}{cc}
  	\frac{1}{4} &0\\
  	0&\frac{1}{12}\\
  \end{array}\right]-\left[\begin{array}{cc}
  \frac{2}{8} &\frac{1}{6}\\
  \frac{2}{4}&\frac{7}{12}\\
\end{array}\right]=\left[\begin{array}{cc}
\frac{1}{2} &-\frac{1}{6}\\
-\frac{1}{2}&\frac{1}{3}\\
\end{array}\right].
$$
As $\mathcal{M}$ is a non-singular M-matrix (the principal minors are positive \cite{fiedler}), by Corollary \ref{cor:exist-est-sol-periodic-model-hopfiels-periodico} this example has a unique 10-periodic solution which is globally exponentially stable.  Figures \ref{fig:modelacao-periodica} gives the plot of the periodic solution of the illustrative numerical example.  Figures \ref{fig:modelacaox1} and \ref{fig:modelacaox2} give the plot of first and second, respectively, component of three solutions of the illustrative numerical example, together with the periodic solution.\\

\begin{figure}[h]
	\centering
	{\includegraphics[width=8.5cm]{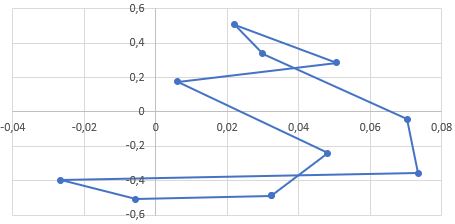}}
	\caption{The 10-periodic solution of the numerical example.
	}\label{fig:modelacao-periodica}
\end{figure}

\begin{figure}[h]
	\centering
	{\includegraphics[width=12.5cm]{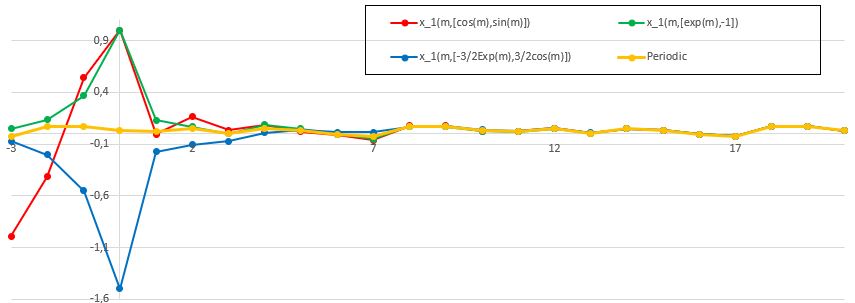}}
	\caption{The first component, $x_1(m)$, of three solutions of the numerical example with initial condition $
		\overline{x}_0(j)=\left(\cos(j),\sin(j)\right),\,j\in[-3,0]_\ze$, $\overline{x}_0(j)=\left(\e^j,-1\right),\,j\in[-3,0]_\ze$, and $\overline{x}_0(j)=\left(-\frac{3}{2}\e^j,\frac{3}{2}\cos(j)\right),\,j\in[-3,0]_\ze$ respectively.
	}\label{fig:modelacaox1}
\end{figure}

\begin{figure}[h]
	\centering
	{\includegraphics[width=12.5cm]{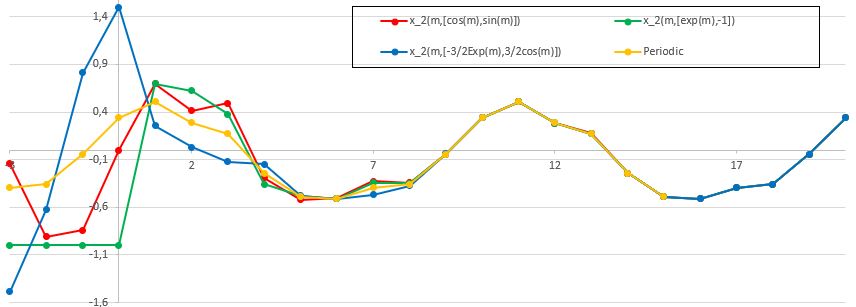}}
	\caption{The second component, $x_2(m)$, of three solutions of the numerical example with initial condition $
		\overline{x}_0(j)=\left(\cos(j),\sin(j)\right),\,j\in[-3,0]_\ze$, $\overline{x}_0(j)=\left(\e^j,-1\right),\,j\in[-3,0]_\ze$, and $\overline{x}_0(j)=\left(-\frac{3}{2}\e^j,\frac{3}{2}\cos(j)\right),\,j\in[-3,0]_\ze$ respectively.
	}\label{fig:modelacaox2}
\end{figure}


\noindent
{\bf Acknowledgments.}\\
This work was partially supported by Funda\c{c}\~ao para a Ci\^encia e a Tecnologia (Portugal) within the Projects UIDB/00013/2020 and UIDP/00013/2020 of CMAT-UM (Jos\'e J. Oliveira) and Project UIDB/00212/2020 of CMA-UBI (Ant\'onio J. G. Bento and C\'esar M. Silva).

\bibliographystyle{elsart-num-sort}

\begin{thebibliography}{10}
\expandafter\ifx\csname url\endcsname\relax
  \def\url#1{\texttt{#1}}\fi
\expandafter\ifx\csname urlprefix\endcsname\relax\def\urlprefix{URL }\fi

\bibitem{aizenberga_aizenberg_hiltner_moraga_Bexten}
I. ~Aizenberg, N. ~Aizenberg, J. ~Hiltner, C. ~Moraga, E. Bexten, Cellular neural networks and computational intelligence in medical image processing, Image and Vision Computing 19 (4) (2001) 177-183.


\bibitem{baldi} P. ~Baldi and A. F. ~Atiya, How delays affect neural dynamics and learning, IEEE Trans. Neural Networks 5 (1994) 612-621. 

\bibitem{Bento-oliveira-Silva}
A.~Bento, J.~Oliveira, C.~Silva, Nonuniform behavior and stability of Hopfield neural networks with delay, Nonlinearity 30 (2017) 3088-3103.

\bibitem{berezansky+braverman+idels} L. Berezansky, E. Braverman, and L. Idels, New global exponential stability criteria for nonlinear delay differential systems with applications to BAM neural networks, Appl. Math. Comput. 243 (2014) 899-910.

\bibitem{cai_mao_wang_yin_karniadakis} S. ~Cai, Z. ~Mao, Z. ~Wang, M. ~Yin, G. ~Karniadakis, Physics-informed neural networks (PINNs) for fluid mechanics: a
review, Acta Mech. Sin. 37(12) (2021) 1727-1738. 

 \bibitem{cichocki} A. ~Cichocki and R. ~Unbehauen,  \emph{Neural Networks for Optimization and Signal Processing}, Wiley, Chichester, 1993.

\bibitem{Chen_Song_Zhao_Liu} X. ~Chen, Q. ~Song, Z. ~Zhao, and Y. Liu, Global $\mu$-stability analysis of discrete-time complex-valued neural
networks with leakage delay and mixed delays, Neurocomputing 175 (2016) 723-735.

 \bibitem{cohen-grossberg} M. ~Cohen and S. ~Grossberg, Absolute stability of global pattern formation and parallel memory storage by competitive neural networks, IEEE Trans. Systems Man Cybernet. 13 (1983) 815--826.  

\bibitem{Dong_Wang_Zhang-AMC-2020} Z. ~Dong, X. ~Wang, and X. ~Zhang, A nonsingular M-matrix-based global exponential stability analysis of higher-order delayed discrete-time Cohen-Grossberg neural networks,  Appl. Math. Comput. 385 (2020) 125401.

\bibitem{dai_du} Z. ~Dai and B. Du, Global dynamic analysis of periodic solution for discrete-time inertial
neural networks with delays, AIMS Math. 6(4) (2021) 3242-3256.

\bibitem{Dong_Wang_Zhang_hu_dinh-NA-2023} Z. ~Dong, X. ~Wang, X. ~Zhang, M. ~Hu, and T. Dinh, Global exponential synchronization of discrete-time higher-order switghed neural networks and its application to multi-channel audio encryption,  Nonlinear Anal. Hybrid Syst. 47 (2023) 101291.





\bibitem{fiedler} M.~Fiedler, Special matrices and their applications in numerical mathematics,
Martinus Nijhoff Publishers, Dordrecht, 1986, translated from the Czech by
Petr P{\v{r}}ikryl and Karel Segeth.


 \bibitem{gopalsamy} K. ~Gopalsamy, Leakage delays in BAM, J. Math. Anal. Appl. 325 (2007) 1117-1132.


\bibitem{Hong_Ma-MBE-2019} Y. Hong and W. Ma, Sufficient and necessary conditions for global attractivity and stability of a class of discrete Hopfield-type neural network with time delays, Mathematical Biosciences and Engineering 16(5) (2019) 4936-4946.

 
\bibitem{hopfield} J.J. ~Hopfield, Neural networks with graded response have collective computational properties like those of two-state neurons, Proc. Natl. Acad. Sci. 81 (1984) 3088-3092.


 \bibitem{kosko} B. ~Kosko, Bidirectional associative memories, IEEE Trans. Systems Man Cybern. 18 (1988) 49-60.  


\bibitem{li+zhang} Y. ~Li and T. ~Zhang, Almost periodic solution for a discrete hematopoiesis model with time delay, 
Int. J. Biomath. 5(1) (2012) 1250003, 9 pp.

 \bibitem{liu} B. Liu, Global exponential stability for BAM neural networks with time-varying delays in the leakage terms, Nonlinear Anal. RWA 14 (2013) 559-566.


\bibitem{lu} H. Lu, Computer-aided diagnosis research of a lung tumor based on a deep convolutional neural network and global features, BioMed. Res. Int.
2021 (2021) Article ID 5513746. 

 \bibitem{marcus+westervelt} C.M. ~Marcus and R.M. ~Westervelt, Stability of analogy neural networks with delay, Phys. Rev. A 39 (1989) 347-359. 

\bibitem{Mohamad_Gopalsamy-AMC-2003} S. ~Mohamad and K. ~Gopalsamy, Exponential stability of continuous-time and discrete-time cellular neural networks with delays, Appl. Math. Comput. 135 (2003) 17-38.


\bibitem{jj2} J. ~Oliveira, Global exponential stability of discrete-time Hopfield neural network models with unbounded delays, J. Difference Equ. Appl. 28(5) (2022) 725-751. 

\bibitem{jj} J. ~Oliveira, Global stability criteria for nonlinear differential systems with infinite delay and applications to BAM neural networks, Chaos Solitons \& Fractals 164 (2022) 112676.

\bibitem{peng} S. ~Peng, Global attractive periodic solutions of BAM neural networks with continuously distributed delays in the leakage terms, Nonlinear Anal. RWA 11 (2010) 2141-2151. 


\bibitem{pham_nguyen_bui_prakash_chapi_bui} B.T. ~Pham, M.D. ~Nguyen, K.-T.T. ~Bui, I. ~Prakash, K. ~Chapi, D.T. ~Bui,  A novel artificial intelligence approach based on multi-layer perceptron neural
network and biogeography-based optimization for predicting coefficient of consolidation of soil, Catena 173 (2019) 302-311.

\bibitem{Qui_Liu_Shu-ADE-2016} S.-B. ~Qiu, X.-G. ~Liu, and Y.-J. ~Shu, A study on state estimation for discrete-time recurrent neural networks with leakage delay and time-varying delay, Advances in Difference Equations  (2016) 234.

\bibitem{Raja_Anthoni-CNSNS-2011} R. ~Raja and S.-M. ~Anthoni, Global exponential stability of BAM neural networks with time-varying delays: The discrete-time case, Commun Nonlinear Sci Numer Simulat  16 (2011) 613-622.



\bibitem{Sowmiya_Raja_Cao_Li_Rajchakit-JFI-20} C. ~Sowmiya, R. ~Raja, J. ~Cao, X. Li, and G. ~Rajchakit, Discrete-time stochastic impulsive BAM neural networks with leakage and mixed time delays: An exponential stability problem, J. Franklin Inst. 355 (2018) 4404-4435.


\bibitem{Sowmiya_Raja_Cao_Rajchakit_Alsaedi-AJC-2020} C. ~Sowmiya, R. ~Raja, J. ~Cao, G. ~Rajchakit, and A. ~Alsaedi, A delay-dependent asymptotic stability criteria for uncertain BAM neural networks with leakage and discrete time-varying delays: A novel summation inequality, Asian Journal Control 22(5) (2020) 1880-1891.

\bibitem{Shumin_Yanhong} S. ~Sun and Y. ~Li, Mean boundedness, global attractivity and almost periodic sequence of stochastic neural networks with discrete-time analogue. Filomat 35(12) (2021) 3919-3931.

\bibitem{Suntonsinsoungvon_Udpin} E. ~Suntonsinsoungvon and S. ~Udpin, Exponential stability of discrete-time uncertain neural networks with multiple time-varying leakage delays, Math. Comput. Simulation 171 (2020) 233-245.

\bibitem{Thoiyab+muruganantham+zhu+gunasekaran} N. Thoiyab, P. Muruganantham, Q. Zhu, and N. Gunasekaran, Novel results on global stability analysis for multiple time-delayed BAM neural networks under parameter uncertainties, Chaos Solitons \& Fractals 152 (2021) 111441. 


\bibitem{velichko_belyaev_boriskov} A. ~Velichko, M. ~Belyaev, and P. ~Boriskov, A model of an oscillatory neural network with multilevel neurons for pattern recognition and computing, Electronics 8 (1) (2019) 75.  

\bibitem{xu_wu} H.~Xu and R. ~Wu, Periodicity, exponential stability of discrete-time neural networks with variable coefficients and delays, Adv. Differ. Equ. (2013) 226.

\bibitem{zheng_du} F. ~Zheng and B. ~Du, Dynamic behaviors of almost periodic solution of discrete-time inertial neural networks with delays, Chinese J. Phys. 73 (2021) 512-522.

\end{thebibliography}

\end{document}